\documentclass{amsart}[12pt]
\usepackage{setspace}
\usepackage{amsmath,amssymb,amsfonts,bm}
\usepackage{amsthm}
\usepackage{amscd}

\newtheorem*{mainthm}{Theorem}
\newtheorem*{maincor}{Corollary}
\newtheorem{thm}{Theorem}[section]
\newtheorem{cor}[thm]{Corollary}

\newtheorem{lemma}[thm]{Lemma}
\newtheorem{prop}[thm]{Proposition}
\newtheorem{defn}[thm]{Definition}

\newtheorem{remark}[thm]{Remark}

\newcommand{\CH}{\mathrm{CH}}

\newcommand{\chat}{\widehat{\mathrm{CH}}}
\newcommand{\X}{\mathfrak{X}}
\newcommand{\Q}{\mathbb{Q}}
\newcommand{\C}{\mathbb{C}}
\newcommand{\Z}{\mathbb{Z}}
\newcommand{\KC}{\mathrm{KC}}
\newcommand{\ie}{\emph{i.e.},\;}
\newcommand{\spec}{\mathrm{Spec}}

\begin{document}
\title{Arithmetic Intersection Theory on Deligne-Mumford Stacks}
\author{Henri Gillet }
\address{Department of Mathematics, Statistics, and Computer Science\\
University of Illinois at Chicago\\
322 Science and Engineering Offices (M/C 249)\\
851 S. Morgan Street\\
Chicago, IL 60607-7045}
\date{\today}
\thanks{Supported by NSF grant DMS-0500762, the Clay Mathematics Institute, and the Fields Institute}
 \begin{abstract}
In this paper the arithmetic Chow groups and their product structure are extended from the category of regular arithmetic varieties to regular Deligne-Mumford stacks proper over a general arithmetic ring. The method used also gives another construction of the product on the usual Chow groups of a regular Deligne-Mumford stack.
    \end{abstract}
\maketitle

\section*{Introduction} Because of the importance of moduli stacks in arithmetic
geometry, it is natural to ask whether the arithmetic intersection
theory introduced in \cite{Gillet-Soule-IHES-arith-intn-thy} can be
extended to stacks.  Indeed arithmetic intersection numbers on stacks and moduli spaces have been studied by a number of authors; see \cite{Kudla-Rap-Yang} and \cite{Brunier-Burgos-Kuhn}, for example.

Recall that the arithmetic Chow theory of \emph{op. cit.} has the following properties.
\begin{enumerate}
    \item $X\mapsto \chat^*(X)$ is a contravariant functor from the category of schemes
    which are regular flat and projective over $S=\mathrm{Spec}(\Z)$, to graded abelian groups
    \item $\chat^*(X)_\mathbb{Q}$ has a functorial graded ring structure
    \item $\chat^1(X)\simeq \widehat{\mathrm{Pic}}(X)$, the group of isomorphism classes
    of Hermitian line bundles on $X$. (A Hermitian line bundle $(L,h)$ on $X$, is a line
    bundle $L$ on $X$ together with a choice of a $C^\infty$ Hermitian metric $h$
    on the associated holomorphic line bundle $L(\mathbb{C})$ over the complex manifold $X(\mathbb{C})$).
    \item Each class in $\chat^p(X)$ is represented by a pair $(\zeta,g_\zeta)$
    with $\zeta=\sum_i[Z_i]$ a codimension $p$ algebraic cycle on $X$ and $g_\zeta$ a ``Green current''
    for $\zeta$, \emph{i.e.} a current of degree $(p-1,p-1)$ on $X(\mathbb{C})$ such that
    $$
    dd^c(g_\zeta)+\delta_\zeta
    $$
    is a $C^\infty$ $(p,p)$-form. Here $\delta_\zeta$ is the $(p,p)$-current
    $\sum_i\delta_{Z_i(\mathbb{C})}$ where $\delta_{Z_i(\mathbb{C})}$ is the current of integration associated to the analytic subspace
    $Z_i(\mathbb{C})\subset X(\mathbb{C})$
    \item There is an exact sequence, for each $p\geq 0$:
\begin{multline*}
\mathrm{CH}^{p,p-1}({X})\overset{\rho}\to H^{2p-1}_\mathcal{D}({X}_\mathbb{R},\mathbb{R}(p))\\
\to\widehat{\mathrm{CH}}^p({X})\to
\mathrm{CH}^p({X})\oplus{Z}^{p,p}(X_\mathbb{R}))
\to\mathrm{H}^{2p}_\mathcal{D}({X}_\mathbb{R},\mathbb{R}(p))
\end{multline*}
Here the $X_\mathbb{R}$ indicates that we are taking real forms on $X(\C)$ on which the anti-holomorphic involution induced by
complex conjugation acts by $(-1)^p$ on $(p,p)$-forms.
    \item There is a theory of Chern classes for Hermitian vector bundles
    $$
    (E,\|\;\|)\mapsto \widehat{C}_p(E,\|\;\|)
    $$
\end{enumerate}
Extending the definition of the Chow groups to stacks is straightforward.
A cycle on a stack is an element of the free abelian group on the set of integral substacks, and rational equivalence is defined similarly.
(See \cite{Gillet-Intersection-theory-on-stacks}.)  To define the arithmetic Chow groups, we must associate Green currents to cycles.
First we show that for a smooth separated stack $\mathfrak X$ over $\mathbb C$
the sheaves of differential forms are acyclic, and that the  groups  $H^q({\mathfrak X}({\mathbb C}),\Omega^p)$ are computed by the global
Dolbeault complexes.  If $\mathfrak X$ is proper, one can show that the Hodge spectral sequence degenerates, and that the cohomology
$H^q({\mathfrak X}({\mathbb C}),{\mathbb C})$ has a Hodge decomposition ``in the strong sense''.
In particular the $\partial\overline{\partial}$-lemma holds.
This allows one to give a definition of the arithmetic Chow groups analogous
to that of \cite{Gillet-Soule-IHES-arith-intn-thy} and \cite{Burgos-JAG-Arith-chow-rings}, though it does not give the product structure.

Before discussing how to define the product on the arithmetic Chow groups of stacks, let us briefly review
how the classical Chow groups and their product structure are defined for stacks. In the 1980's two different approaches to intersection theory on
stacks over fields were introduced; the first in \cite{Gillet-Intersection-theory-on-stacks} was via Bloch's
formula. This approach has not been applied to the arithmetic Chow groups of schemes,
never mind stacks, in part because Bloch's formula depends on Gersten's conjecture, which is not known for general regular schemes.
The other construction of intersection theory on stacks was by
Vistoli \cite{Vistoli-Intersection-thy-on-stacks}, using ``Fulton
style'' intersection theory, and in particular operational Chow groups. While Fulton's operational Chow groups make sense for regular schemes,
it is not clear how to construct the arithmetic Chow groups of schemes (or stacks), and their product structure, operationally.
The problem with both of these approaches is that the product on the arithmetic Chow
groups depends on the $*$- product for the Green current of two cycles. which is only
defined when the cycles intersect properly, and thus the moving lemma plays a key role in
the construction of the intersection product on the arithmetic Chow groups.  Note however that
combining the method of Hu \cite{J-Hu} with  Kresch's approach in \cite{Kresch} to the construction of the intersection
product on stacks over fields, via deformation to the normal cone, might provide a way around this problem.

In this paper, we shall use a construction that is a variant on the operational method.
It was first observed by Kimura \cite{Kimura} (see also \cite{Bloch-gillet-soule}) that if $X$ is a possibly singular variety proper over a field of characteristic zero then its operational Chow groups can be computed using hypercovers.
It follows from this result that for a proper variety $X$ over a field  of characteristic zero,
the operational Chow groups of $X$ are isomorphic to the inverse limit of $\mathrm{CH}^*(Y)$ over the category all surjective morphisms  $Y\to X$ with $Y$ smooth and projective.  This suggests using a similar construction for the arithmetic Chow groups of stacks.

Suppose  for a moment that we have extended the functor $\chat^*$ to the category of separated stacks over a fixed base $S$.
Then for each $p:V\to \mathfrak{X}$, with $p$ proper and
surjective, and $V$ a regular quasi-projective variety over $S$ we will have a natural
homomorphism
$p^*:\widehat{\mathrm{CH}}^*(\mathfrak{X})\to\widehat{\mathrm{CH}}^*(V)$
and hence a homomorphism
$$\widehat{\mathrm{CH}}^*(\mathfrak{X})\to \lim_{\overset{\longleftarrow}{p:V\to\mathfrak{X}} } \widehat{\mathrm{CH}}^*(V)$$
Since we already have well defined functorial products on the groups
$\widehat{\mathrm{CH}}^*(V)$, it follows that
$\underset{\leftarrow}{\lim}\, \widehat{\mathrm{CH}}^*(V)$ has a
natural product structure, which is contravariant with respect to
$\mathfrak{X}$.  A Hermitian vector
bundle   $\overline{E}=(E,h)$ on $\mathfrak{X}$ has Chern classes in
$\underset{\leftarrow}{\lim}\, \widehat{\mathrm{CH}}^*(V)$,
since the bundle pulls back to any $V$ over
$\mathfrak{X}$.
Note that a similar construction appears in \cite{Brunier-Burgos-Kuhn}, where they consider towers of Shimura varieties with level structures, rather than the underlying stack.

Now the key point is that, even though we do not, \emph{a priori}, have products and
pull-backs on $\widehat{\mathrm{CH}}^*(\mathfrak{X})$, we have:

\begin{mainthm} If $\X$ is regular Deligne-Mumford stack  which is flat and proper over
a base $S=\spec(\mathcal{O}_K)$ for $\mathcal{O}_K$ the
ring of integers in a number field $K$,
there is a canonical isomorphism
$$\lim_{\overset{\longleftarrow}{p:V\to\mathfrak{X}} } \widehat{\mathrm{CH}}^*(V)_\Q\to \widehat{\mathrm{CH}}^*(\mathfrak{X})_\Q $$
\end{mainthm}

The idea of the proof is to show that the appropriate variants of this statement are true
for differential forms, and also for both the usual Chow groups (tensored with $\mathbb Q$) and for cohomology.

\begin{maincor}
There is a product structure on
$\widehat{\mathrm{CH}}^*(\mathfrak{X})_\Q$ which is  functorial in
$\mathfrak{X}$, and the theory of  Chern classes for Hermitian vector bundles on arithmetic varieties extends to
Hermitian vector bundles on stacks over $S$.
\end{maincor}

Following preliminaries on the Dolbeault cohomology of stacks in section 1,  in section \ref{G-K-stacks} we discuss the
$G$-theory and $K$-theory of stacks. We show that the isomorphism (due to Grothendieck)
between the graded vector space associated to the
$\gamma$-filtration on $K_0$ and the Chow groups of a regular scheme also holds for regular Deligne-Mumford stacks.
This gives yet another construction of the product on the Chow groups, with rational coefficients, of a stack.
We then go on to show that the motivic weight complex of a regular \emph{proper}
Deligne-Mumford stack over $S$ is (up to homotopy) concentrated in degree zero.
Thus to each regular proper Deligne-Mumford stack we can associate a pure motive.
This extends a result of Toen for varieties over perfect fields. We then prove the main
theorem in section \ref{ch-hat-proper-stacks}. Finally in section 5 we consider the case of non-proper stacks.

Throughout the paper we shall fix a base $S$ which is the spectrum of the ring of integers in a number field or
more generally of an arithmetic ring in the sense of \cite{Gillet-Soule-IHES-arith-intn-thy}.
In particular a variety (over $S$) will be an integral scheme which is separated and of finite type over $S$.

I would like to thank the referee for a very careful reading of the
manuscript and for many valuable comments and questions.

\section{Dolbeault Cohomology of Stacks}\label{Dolbeault}

Let $\mathfrak{X}$ be a regular Deligne-Mumford stack of finite type over
the field of complex numbers $\mathbb{C}$.  Let
$\mathfrak{X}(\mathbb{C})$ be the associated smooth stack in the
category of complex analytic spaces.

Recall that $\mathfrak{X}$ is a category cofibered in groupoids
over the category of algebraic spaces; since $\mathfrak{X}$ is a
Deligne-Mumford stack, we will, equivalently, view $\mathfrak{X}$
as cofibered over the category of schemes over $\mathbb{C}$. A
1-morphism from a scheme $U$ to $\mathfrak{X}$ is really an object
in the fiber of $\mathfrak{X}$ over $U$.  A morphism between two
objects in the category $\mathfrak{X}$ is \'etale (respectively
surjective) if the morphism of the associated schemes is \'etale
(respectively surjective), and a morphism  $p:U \to \mathfrak{X}$
from a scheme  to $\mathfrak{X}$ is \'etale  (respectively
surjective) if for every morphism $f:X\to \mathfrak{X}$ with
domain a scheme, the projection $f\times_\mathfrak{X} p\to f$ is
\'etale  (respectively surjective).

Following the discussion in section 12.1 of
\cite{Laumon-M-B-stacks} and definition 4.10 of \cite{deligne-mumford},
we can take the \'etale site of
$\mathfrak{X}$ to consist of all schemes $p:U\to \mathfrak{X}$
\'etale over $\mathfrak{X}$, with covering families  consisting of
those families of morphisms in the category $\mathfrak{X}$ for
which the associated family of morphisms of schemes is a covering
family for the \'etale topology.

Because every Deligne-Mumford stack admits an \'etale cover
$\pi:U\to \mathfrak{X}$ by a scheme, to give a sheaf of sets $F$
on the \'etale site of $\mathfrak{X}$ is equivalent to giving a
sheaf $F_U$ together with an isomorphism between the two
pull-backs of $F_U$ to $U\times_\mathfrak{X} U$ satisfying the
cocycle condition of \cite{Laumon-M-B-stacks} 12.2.1.

Similarly, to give a sheaf $F$ on the stack
$\mathfrak{X}(\mathbb{C})$ over the category of analytic spaces,
is equivalent to giving a sheaf  $F_U$ for any \'etale cover
$\pi:U\to \mathfrak{X}$  in the classical topology on
$U(\mathbb{C})$ together with an isomorphism between the two
pull-backs of $F_U$ to
$U(\mathbb{C})\times_{\mathfrak{X}(\mathbb{C})} U(\mathbb{C})$
satisfying the cocycle condition.

Note that if $f:U\to V$ is an \'etale map between complex analytic manifolds, then we have \emph{isomorphisms}, for all $(p,q)$:
$$
f^*:f^{-1}\mathcal{A}^{p,q}_U \to \mathcal{A}^{p,q}_V
$$
where $\mathcal{A}^{p,q}_V$ is the usual sheaf of $(p,q)$-forms.
We therefore have the sheaf $\mathcal{A}^{p,q}_\mathfrak{X}$ of
differential forms of type $(p,q)$ on the stack
$\mathfrak{X}(\mathbb{C})$, together with the $\partial$ and
$\overline{\partial}$ operators.  Notice that if
$p:U\to\mathfrak{X}$ is \'etale, then the group of automorphisms
of $p$ over $U$ acts \emph{trivially} on $\mathcal{A}^{p,q}_U$.

The total de Rham complex of $\mathfrak{X}$ is a resolution of the
constant sheaf $\mathbb{C}$, since this can be checked locally in
the \'etale topology, and similarly the complexes
$(\mathcal{A}^{p,*}_\mathfrak{X},\overline{\partial})$ are
resolutions of the sheaves of holomorphic $p$-forms
$\Omega^p_\mathfrak{X}$.  We have the usual Hodge spectral
sequence:
$$E_1^{p,q}(\mathfrak{X})=H^q(\mathfrak{X},\Omega^p)\Rightarrow H^{p+q}(\mathfrak{X},\mathbb{C})\;.$$

If $U$ is a complex manifold, we also have for each $(p,q)$, the sheaf $\mathcal{D}^{p,q}_{U}$ of $(p,q)$-forms with distribution coefficients; if $V\subset U$,  $\mathcal{D}^{p,q}_{U}(V)$ is the bornological dual ${D}^{p,q}(V)$ of the Frechet space $A^{n-p,n-q}_c(V)$ of compactly supported
$(n-p,n-q)$-forms on $V$, where $n=\mathrm{dim}(V)$.  If $p:V\to V'$ is \'etale, we have a push forward map $p_*:A^{n-p,n-q}_c(V)\to A^{n-p,n-q}_c(V')$ which is a continuous map of Frechet spaces inducing a pull back map
$p^*:D^{p,q}(V')\to D^{p,q}(V)$, thus for
$\mathfrak{X}$ a Deligne-Mumford stack over $\mathbb{C}$  we get a sheaf $\mathcal{D}^{p,q}_{\mathfrak{X}}$ of $(p,q)$-forms with
distribution coefficients on $\mathfrak{X}(\mathbb{C})$.  There is a natural inclusion
$\mathcal{A}^{*,*}_{\mathfrak{X}}\subset\mathcal{D}^{*,*}_{\mathfrak{X}}$, the operators $\partial$ and
$\overline{\partial}$ extend to $\mathcal{D}^{*,*}_{\mathfrak{X}}$, and the inclusion is a quasi-isomorphism of double complexes.
It follows that the Hodge spectral sequence also arises as the cohomology spectral sequence of the filtered complex associated to the
double complex of sheaves $\mathcal{D}^{*,*}_{\mathfrak{X}}$ with the Hodge filtration.

Note that in general on a smooth Artin or Deligne-Mumford stack
over the complex numbers, the sheaves
$\mathcal{A}^{*,*}_\mathfrak{X}$ and $\mathcal{D}^{*,*}_{\mathfrak{X}}$ need not be acyclic. For example
this is not the case for $B\mathbb{G}L_{n,\mathbb{C}}$, nor for the affine line with the origin doubled.
However we have:
\begin{prop}\label{acyclic}
If $\mathfrak X$ is a smooth \emph{separated} Deligne-Mumford stack over $\mathbb C$, then the sheaves
$\mathcal{A}^{*,*}_\mathfrak{X}$ and $\mathcal{D}^{*,*}_{\mathfrak{X}}$ are acyclic.
\end{prop}
\begin{proof}
We know (\cite{Laumon-M-B-stacks} 19.1, \cite{deligne-rapoport} Ch. I. section 8, and \cite{Keel-Mori})
that $\mathfrak X$ has a ``coarse space'' $X$ which is a separated algebraic space, and for which the map
$\pi:{\mathfrak X}\to X$ is finite and induces a bijection on (isomorphism classes of) geometric points.
Furthermore, given an \'etale cover $p:P\to\mathfrak{X}$, a closed point
$\xi:\mathrm{Spec}(\mathbb{C})\to \mathfrak{X}$ and a lifting $x:\mathrm{Spec}(\mathbb{C})\to P$ of $\xi$,
the map $\pi^*$ induces an isomorphism between $\mathcal{O}_{X,\pi(\xi)}^h$ and the invariants
$\mathcal{O}_{P,x}^h/G$ of the action of the (finite) group of automorphisms of $\xi$ on the Henselization of the local ring of $P$ at $x$
(see \cite{deligne-rapoport} Ch. I. 8.2.1, and \cite{Laumon-M-B-stacks}, 6.2.1).
Since $\mathfrak X$ is \emph{separated}, it is straightforward to check that this isomorphism extends to an open neighborhood $U$ of
$x$ in the analytic topology, giving an isomorphism $[U/G]\to\pi^{-1}(V)$ where $V\subset X$ is an open neighborhood of
$\pi(\xi)$. (See
II.5.3 of \cite{Bogaart},  which is based on 2.8 of \cite{Vistoli-Intersection-thy-on-stacks}.) Since $[U/G]=\pi^{-1}(V)$ is a finite \'etale groupoid,
a standard transfer argument shows that the restrictions of the sheaves
$\mathcal{A}^{*,*}_\mathfrak{X}$ and $\mathcal{D}^{*,*}_{\mathfrak{X}}$ to $[U/G]$ are acyclic.
Now take a cover of the analytic space $X(\mathbb{C})$ by open sets $V$ such that $\pi^{-1}(V)$ is a finite \'etale groupoid,
and using a partition of unity subordinate to this cover, the proof finishes by a standard argument.
\end{proof}
We immediately obtain:
\begin{cor}
Let $\mathfrak{X}$ be a Deligne-Mumford stack smooth and separated over the complex numbers $\mathbb{C}$. Since the sheaves  $\mathcal{D}^{*,*}_{\mathfrak{X}}$ and $\mathcal{A}^{*,*}_{\mathfrak{X}}$ are acyclic, the groups $H^*(\mathfrak{X},\mathbb{C})$ and $H^*(\mathfrak{X},\Omega^*)$ are computed by the global de Rham and Dolbeault complexes of $\mathfrak{X}(\mathbb{C})$ respectively.
\end{cor}

If $f:\mathfrak{X}\to\mathfrak{Y}$ is a proper, representable, morphism of relative dimension $d$ between regular stacks over $\mathbb{C}$, there is a push forward map:
$$
f_*:f_*(\mathcal{D}^{p,q}_{\mathfrak{X}})\to \mathcal{D}^{p-d,q-d}_{\mathfrak{Y}}
$$
which is  defined as follows.
If $p:U\to \mathfrak{Y}$ is an \'{e}tale morphism from a
scheme to $\mathfrak{Y}$, let
$\pi_U: \mathfrak{X}\times_\mathfrak{Y} U \to U$ be the induced proper morphism of schemes.  Suppose that $T\in
{D}^{p,q}((\mathfrak{X}\times_\mathfrak{Y} U)(\mathbb{C}))$; then if  $\phi\in A^{n-p,n-q}_c(U)$, where $n$ is the dimension of $\mathfrak{Y}$, is a  compactly supported form, we have
that $\pi_U^*(\phi)$ is a compactly supported form on $(\mathfrak{X}\times_\mathfrak{Y} U)(\mathbb{C})$, and
hence we can define $\pi_*(T)$ to be the current in ${D}^{p-d,q-d}(U(\mathbb{C}))$ such that  $\pi_*(T)(\phi):= T(\pi_U^*(\phi))$.

\begin{lemma}\label{pull-back-push-for-forms} Let
$f:\mathfrak{X}\to\mathfrak{Y}$ be a proper, representable, morphism  between   smooth
Deligne-Mumford stacks over the complex numbers $\mathbb{C}$.  Suppose that $f$ is generically finite of degree $d$.
Then the composition of the maps
$$
\mathcal{A}^{p,q}_{\mathfrak{Y}}
\overset{f^*}{\to}
f_*\mathcal{A}^{p,q}_{\mathfrak{X}}
\to
f_*\mathcal{D}^{p,q}_{\mathfrak{X}}
\overset{f_*}{\to}
\mathcal{D}^{p,q}_{\mathfrak{Y}}
$$
is  $d$ times the natural inclusion $\mathcal{A}^{p,q}_{\mathfrak{Y}}\to
\mathcal{D}^{p,q}_{\mathfrak{Y}}$.
\end{lemma}
\begin{proof}
Suppose that $p:U\to \mathfrak{Y}$ is \'etale, with $U$ a scheme, and that
$\alpha\in A^{p,q}({U}(\mathbb{C}))$. Without loss of generality,
we can suppose that $U$ is
irreducible, so that $f_U:f^{-1}(U)\to U$ has a well defined degree. Since $f_U^*(\alpha)$ is a smooth form $f_{U,*}f_U^*(\alpha)$ is a current which is represented by a
locally $L^1$ form, and hence is determined by its value on the
complement of any subset of $U(\mathbb{C})$ with measure zero. But
outside a set of measure zero on $U(\mathbb{C})$, the map of complex
manifolds $f_U:f^{-1}(U)(\mathbb{C})\to U(\mathbb{C})$ is a finite
\'{e}tale cover, and hence $f_{U,*}f_U^*(\alpha)=d\alpha$, where $d$
is the degree of $\pi$.
\end{proof}
Since the inclusions $\mathcal{A}^{p,*}_{\mathfrak{X}}\subset\mathcal{D}^{p,*}_{\mathfrak{X}}$ induce quasi-isomorphisms on the sheaves of Dolbeault complexes, we get:
\begin{cor}\label{Coh-splits} Let
$f:\mathfrak{X}\to\mathfrak{Y}$ be a proper, representable, generically finite morphism  between   smooth Deligne-Mumford stacks
over the complex numbers $\mathbb{C}$.  Then we have split injections:
$$f^*:H^*(\mathfrak{Y},\mathbb{C})\to H^*(\mathfrak{X},\mathbb{C}) $$
and
$$f^*:H^*(\mathfrak{Y},\Omega_\mathfrak{Y}^*)\to H^*(\mathfrak{X},\Omega_\mathfrak{X}^*) $$
compatible with the Hodge spectral sequence.
\end{cor}
\begin{prop} Let $\mathfrak{X}$ be a Deligne-Mumford stack smooth and proper over the complex numbers $\mathbb{C}$.
The Hodge spectral sequence for $\mathfrak{X}(\mathbb{C})$ degenerates at $E_1$, and the cohomology groups $H^*(\mathfrak{X},\mathbb{C})$ admit a Hodge decomposition ``in the strong sense'' (see \cite{Peters_Steenbrink} Definition 2.24).
\end{prop}
\begin{proof} Since $\mathfrak{X}$ is proper over $\mathbb{C}$ it is in particular separated.  Hence by Chow's lemma \cite{Laumon-M-B-stacks} 16.6.1, we know that there exists a proper surjective and generically finite map
$\pi:X\to \mathfrak{X}$ with $X$ smooth and proper over $\mathbb{C}$.   Since the Hodge spectral sequence for ${X}(\mathbb{C})$ degenerates, by corollary \ref{Coh-splits} we know that the same is true for $\mathfrak{X}(\mathbb{C})$.  Furthermore since the $\mathcal{A}^{*,*}_\mathfrak{X}$ are acyclic by proposition \ref{acyclic}, the pull back map  is induced by the map of complexes:
$f^*:A^{*,*}(\mathfrak{X}(\mathbb{C}))\to A^{*,*}(X(\mathbb{C}))$, and is invariant under complex conjugation. Hence for $p+q=n$, $F^pH^n(\mathfrak{X}(\mathbb{C}),\mathbb{C})
\cap\overline{ F^{q+1}H^n(\mathfrak{X}(\mathbb{C}),\mathbb{C})}=\{0\}$.
\end{proof}
The proposition and the fact that the Hodge cohomology of $\mathfrak{X}(\mathbb{C})$ is computed by the
Dolbeault complexes $A^{*,*}(\mathfrak{X}(\mathbb{C}))$ give us:
\begin{cor}
The $\partial\overline{\partial}$-lemma holds for
$\mathfrak{X}(\mathbb{C})$.
\end{cor}
\begin{proof}
See \cite{Peters_Steenbrink} 2.27 and 2.28.
\end{proof}
This proof is a variation on an argument for complex manifolds, by which one deduces a strong Hodge decomposition and the
$\partial\overline{\partial}$-lemma for Moishezon manifolds (see, for example, section 9 of Demailly's article in \cite{Bertin_Demailly_Illusie_Peters}).

We shall need the following lemma later.
\begin{lemma}\label{inverse-limit-of-forms}
If $\X$ is a smooth separated Deligne-Mumford stack over $\C$, the natural map:
$$A^{*,*}(\X)\to \lim_{\overset{\longleftarrow}{p:X\to\mathfrak{X}} }A^{*,*}(X)$$
where the inverse limit is over all proper surjective maps with domain a smooth variety,
is an isomorphism.
\end{lemma}
\begin{proof}
From lemma \ref{pull-back-push-for-forms} we know the map is injective. We therefore need only show that it is surjective.
Suppose then that $(p:X\to \mathfrak{X})\mapsto \alpha_p$ is an element in the inverse limit.

First note that  by Zariski's main theorem (\cite{Laumon-M-B-stacks}, 16.5) any \'etale map  $\pi:U\to \X$
from a smooth variety to $\X$ extends
to a finite map $V\to \X$, and hence by resolution of singularities to a proper surjective map $\tilde{\pi}:\tilde{V}\to\X$.
Thus, restricting $\alpha_{\tilde{\pi}}$, we get a form on $U$. A priori this depends depends on the choice of the
factorization of $\pi$ through $\tilde{\pi}$. However given two such factorizations,
$U\subset\tilde{V}_1\to\X$ and $U\subset\tilde{V}_2\to\X$ we can factor through the fiber product
$\tilde{V}_1\times_\X\tilde{V}_2$ and applying resolution of singularities again we see that the form is indeed independent of the factorization.
Writing $\alpha_\pi$ for this form, by a slight variation of the the preceding argument, we also see that
$\alpha_\pi$ is contravariant with respect to $\pi$.
Hence  we get an element in the inverse limit
of $A^{*,*}(U)$ over all \'etale maps $U\to \mathfrak{X}$, \i.e. an element  $\alpha\in A^{*,*}(\mathfrak{X})$.

Now we want to show that for any $p:X\to \mathfrak{X}$, $\alpha_p=p^*(\alpha)$. Let $\pi:P\to \mathfrak{X}$ be \'etale and surjective with $P$ a smooth scheme over $\mathbb{C}$.  Let $i:P\to \bar{P}, \bar{\pi}:\bar{P}\to\mathfrak{X}$ be a factorization
of $\pi$ through a finite map with $P$ dense in $\bar{P}$.
Since $P$ is dense in $\bar{P}$ $\alpha_{\bar{\pi}}=\bar{\pi}^*\alpha$, with $\bar{P}$ smooth over $\mathbb C$.
Now by resolution of singularities, we know that there exists a proper
morphism $Q\to \bar{P}\times_\mathfrak{X} X$, with $Q$ smooth over $\mathbb{C}$ such that
the induced maps $f:Q\to \bar{P}$, $g:Q\to X$ are surjective.  Now $\alpha_Q=(f\cdot\bar{\pi})^*(\alpha)$.
It follows that $g^*(\alpha_p)=f^*(\alpha_{\bar{\pi}})=(g\cdot p)^*\alpha=g^*(p^*(\alpha))$.
However, since $g$ is a surjective map between smooth varieties over $\mathbb{C}$, $g^*$ is injective, and hence $\alpha_p=p^*(\alpha)$.
\end{proof}

Given a codimension $p$ cycle $\zeta=\sum_i [Z_i]$ on a stack
$\mathfrak X$, we know, using the local to global spectral sequence for
cohomology (in the \'etale topology) with supports in $|\zeta|$,  that
the associated current $\delta_\zeta$ represents the cycle class
$[\zeta]\in H^{2p}(\mathfrak{X},\mathbb{R}(p))$.  Hence
as in \cite{Gillet-Soule-IHES-arith-intn-thy}, if we choose an arbitrary $C^\infty$
$(p,p)$-form $\omega$ representing this class, then
by the $\partial\overline{\partial}$-lemma, we know that there is a current
$g\in D^{(p-1,p-1)}(\mathfrak{X}(\mathbb{C}))$ such that
$$dd^c(g)+\delta_\zeta=\omega.$$

We can compute the real Deligne Cohomology of a Deligne-Mumford stack which is proper and smooth over $\mathbb{C}$
using the Deligne complexes of Burgos associated to the Dolbeault complex of $\mathfrak{X}(\C)$ (see \cite{Burgos-Duke}). However,
we do not know how to extend the results of Burgos to non-proper stacks, and therefore we do not know how to give a
cohomological construction of Green currents analogous to that of \textit{op. cit}.

\section{Chow groups and the $K$-theory of stacks.}\label{G-K-stacks}
\subsection{Chow groups of Deligne-Mumford Stacks}
Recall the definition of cycles on, and Chow groups of, stacks from \cite{Gillet-Intersection-theory-on-stacks}.
\begin{defn}\mbox{}\\
\noindent{\textbf 1.} If $\mathfrak{X}$ is a stack, the group of
codimension $p$ cycles $Z^p(\mathfrak{X})$ is the free abelian group
on the set of codimension $p$ reduced irreducible substacks of
$\mathfrak{X}$. This is isomorphic to the group of codimension $p$
cycles on the coarse space $|\mathfrak{X}|$.

\noindent{\textbf 2.} If $\mathfrak{W}$ is a integral stack, write
$\mathbf{k}(\mathfrak{W})$ for its function field.  This may be
defined as \'etale $H^0$ of the sheaf of total quotient rings on
$\mathfrak{W}$; and if $f\in \mathbf{k}(\mathfrak{W})$, we may
define $\mathrm{div}(f)$ locally in the \'etale topology. Therefore,
if $\mathfrak{W}\subset \mathfrak{X}$ is a codimension $p$ integral
substack, and $f\in \mathbf{k}(\mathfrak{W}) $, we have
$\mathrm{div}(f)\in Z^{p+1}(\mathfrak{X})$.

\noindent{\textbf 3.}  We define $\mathrm{CH}^p(\mathfrak{X})$ to be
the quotient of $Z^p(\mathfrak{X})$ by the subgroup consisting of
divisors of rational functions on integral subschemes of codimension
$p-1$.

\noindent{\textbf 4.}
We define $\mathrm{CH}^{p,p-1}(\mathfrak{X})$ as in section 3.3.5 of \cite{Gillet-Soule-IHES-arith-intn-thy}, to be the homology of the complex
$$\bigoplus_{x\in\X^{(p-2)}}K_2(\mathbf{k}(x))\to\bigoplus_{x\in\X^{(p-1)}}\mathbf{k}(x)^*\to Z^p(\X) $$
at the middle term. These groups are isomorphic to the groups of section 4.7 of \cite{Gillet-Intersection-theory-on-stacks},
though there they are written $\mathrm{CH}^{p-1,p}(\mathfrak{X})$; this follows from the argument of the proof of theorem 6.8 of \emph{op. cit.}
\end{defn}
Note that this definition does not require that the Chow groups have rational coefficients.
However once we pass to the \'etale site of $\mathfrak X$, we will need to tensor with $\mathbb Q$.
\subsection{$G$-theory of stacks}
Let us review the basic facts about the $G$-theory of Deligne-Mumford stacks, following section 4.2 of
\cite{Gillet-Soule-Wt-Cplxs-Arith-Vartys}, where details and proofs of the following may be found.

Let ${\mathbf{G}}$ be the presheaf of spectra  on the \'etale site induced by the functor
$\mathbf{G}$. We shall take this functor with rational coefficients, \emph{i.e.},
we take the usual $\mathbf{G}$-theory functor to the category of symmetric spectra, and take the
smash product with the Eilenberg-Maclane spectrum $H\mathbb{Q}$.
It is a (by now) elementary observation, first made by  Thomason (\emph{c.f.} \cite{AKTEC})
that if $X$ is a scheme,
the natural map on \emph{rational} $G$-theory:
$$\mathbf{G}({X})\to \mathbf{H}_{\text{\'et}}({X},\mathbf{G})$$
is a weak equivalence. Hence we can extend rational $G$-theory from schemes to stacks by defining:
\begin{defn}
If $\mathfrak{X}$ is a stack, we define the $G$-theory spectrum $\mathbf{G}(\mathfrak{X})$ to be
$\bm{H}_{\text{\'et}}(\mathfrak{X},\mathbf{G})$.
\end{defn}

Notice that it follows from the definition, and Thomason's result, that
if $\pi:V.\to \mathfrak{X}$ is an \'{e}tale hypercover of a
Deligne-Mumford  stack, the natural map:
$$\mathbf{G}(\mathfrak{X})\to \mathrm{holim}_i(\mathbf{G}(V_i))$$
is a weak equivalence.

It follows immediately that if $f:\mathfrak{Y}\to\mathfrak{X}$ is a \emph{representable} proper morphism of stacks, then
there is a natural map $f_*:\mathbf{G}(\mathfrak{Y})\to \mathbf{G}(\mathfrak{X})$ (this was already observed in
\cite{Gillet-Intersection-theory-on-stacks}). Similarly, if $f$ is a \emph{flat} representable morphism, there is
a natural pull-back map  $f^*:\mathbf{G}(\mathfrak{X})\to \mathbf{G}(\mathfrak{Y})$.
If $f:\mathfrak{Y}\to\mathfrak{X}$ is a closed substack, with complement $j : {\mathfrak U} \to {\mathfrak X}$,
it is then straightforward to check that Quillen's localization theorem extends to stacks, \ie
$$
{\mathbf G}({\mathfrak Y}) \overset{i_*}{\longrightarrow} {\mathbf G}({\mathfrak X}) \overset{j^*}{\longrightarrow} {\mathbf G}({\mathfrak U})
$$
is a fibration sequence.

In order to define a pushforward for non-representable morphisms, in \cite{Gillet-Soule-Wt-Cplxs-Arith-Vartys},
we replace stacks by simplicial varieties, as follows. First observe, that using the strict covariance of our model for
$G$-theory, we can extend the functor $\mathbf{G}$ from schemes to simplicial schemes with proper face maps.

Let  $p : X. \to \mathfrak{X}$ be a proper morphism to a stack from  a simplicial variety with proper face maps.
We construct, in \emph{op. cit.}, a map $p_*:\mathbf{G}(X.)\to \mathbf{G}(\mathfrak{X})$, which is well defined in the
rational stable homotopy category and is compatible with composition, in the sense that
if $f.:Y.\to X.$ is a map of simplicial varieties, we have that
$$ (p\cdot f)_*=p_*\cdot f_*:\mathbf{G}(Y.)\to  \mathbf{G}(\mathfrak{X})$$
in the rational stable homotopy category.
This construction gives an extension of $\mathbf{G}$-theory
from simplicial varieties to stacks, because we have:
\begin{lemma}
Suppose that $p:X.\to {\mathfrak X}$ is a proper hypercover.
Then $p_*:\mathbf{G}(X.)\to \mathbf{G}(\mathfrak{X})$ is a weak equivalence.
\end{lemma}
From which we then get:
\begin{thm}\label{stacks}
Let $f : \mathfrak{X} \to \mathfrak{Y}$ be a proper, not necessarily representable, morphism of stacks.
Then there exists a canonical map (in the homotopy category)
$$\mathbf{G}(\mathfrak{X})\to \mathbf{G}(\mathfrak{Y})$$ with the property that for any commutative square:
$$
\begin{CD}
X. @>\tilde{f}>> Y.\\
@V{p.}VV   @VV{q.}V\\
\mathfrak{X} @>{f}>> \mathfrak{Y}
\end{CD}
$$
in which $p.$ and $q.$ are proper morphisms with domains simplicial schemes with proper face maps,
we have a commutative square in the stable homotopy category:

$$
\begin{CD}
\mathbf{G}(X.) @>\tilde{f_*}>> \mathbf{G}(Y.)\\
@V{\mathbf{G}(p.)}VV   @VV{\mathbf{G}(q.)}V\\
\mathbf{G}(\mathfrak{X}) @>{\mathbf{G}(f)}>> \mathbf{G}(\mathfrak{Y})
\end{CD}
$$
\end{thm}

It is shown in \cite{Gillet-Soule-Wt-Cplxs-Arith-Vartys} that the natural (in general non-representable) map from a separated Deligne-Mumford
stack to its coarse space
induces an isomorphism on $\mathbf{G}$-theory.

\begin{lemma}\label{surjective}
Let $\mathfrak{X}$ be a Deligne-Mumford stack, and suppose that $\pi:X\to \mathfrak{X}$ is a proper and surjective map,
with $X$ a variety.  Then $\pi_*:G_0(X)\to G_0(\mathfrak{X})$ is surjective.  (Recall that we are using $\mathbb Q$ coefficients.)
\end{lemma}
\begin{proof}
First, observe that if $\xi$ is a generic point of $\mathfrak{X}$, the stack $\xi$ is equivalent to a
finite \'etale groupoid acting on a field $F$, and $G_*(\xi)$ is a
direct summand of $G_*(F)$ (see \cite{Gillet-Intersection-theory-on-stacks}).
Hence if $F\subset E$ is any finite extension of $F$, the direct image, or transfer,
map $G_*(E)\to G_*(\xi)$ is surjective, since we are using $G$-theory with rational coefficients). Hence if $f:X\to\xi$ is
a proper morphism from a scheme to $\xi$, it also induces a surjective map on $G$-theory.
The lemma now follows by localization and noetherian induction.
\end{proof}

\subsection{$K$-theory of stacks}
Just as we defined the $G$-theory of a stack as the hypercohomology of the $G$-theory sheaf, we can
consider the presheaf $\mathbf K$ of $K$-theory spectra  which associates to $X$ the $K$-theory spectrum (again with rational coefficients) of the
category of locally free sheaves on the big Zariski site, and then define:
$${\mathbf K}(\mathfrak{X}):=\mathbf{H}_{\text{\'et}}(\mathfrak{X},\mathbf{K})$$
This gives a contravariant functor from stacks to ring spectra.  The tensor product between locally free and coherent sheaves induces a pairing
$${\mathbf K}(\mathfrak{X})\wedge {\mathbf G}(\mathfrak{X})\to {\mathbf G}(\mathfrak{X}) .$$
\begin{lemma}
If $f:\mathfrak{X}\to \mathfrak{Y}$ is a proper representable morphism between stacks,
the projection formula holds, \ie the following diagram is commutative in the stable homotopy category
$$
\begin{CD}
\mathbf{K}(\mathfrak{Y})\wedge \mathbf{G}(\mathfrak{X})
    @>{1\wedge f_*}>> \mathbf{K}(\mathfrak{Y})\wedge \mathbf{G}(\mathfrak{Y}) @>>>
    \mathbf{G}(\mathfrak{Y}) \\
@V{f^*\wedge 1}VV  @.  @AA{f_*}A\\
\mathbf{K}(\mathfrak{X})\wedge \mathbf{G}(\mathfrak{X}) @>>> \mathbf{G}(\mathfrak{X}) @= \mathbf{G}(\mathfrak{X})
\end{CD}
$$
\end{lemma}
\begin{proof}
Since $f$ is representable, if $\pi:V.\to \mathfrak{Y}$ is any \'etale hypercover,
$W.=V.\times_\mathfrak{Y}\mathfrak{X} \to \mathfrak{X}$ is also an \'etale hypercover, and so the result follows from the same statement for
the morphisms $W_i\to V_i$ of schemes. For schemes the assertion follows from projection formula for the
tensor product of locally free and flasque quasi-coherent sheaves, which is true up to canonical isomorphism.
\end{proof}

Combining lemma \ref{surjective} with the previous lemma, we get:
\begin{cor}
If $\pi:{X}\to \mathfrak{Y}$ is a proper surjective morphism from a regular
variety to a regular Deligne-Mumford stack,
then $\pi^*:K_0(\mathfrak{Y})\to K_0({X})$ is a split injection.
\end{cor}

\subsection{Operations on the $K$-theory of stacks}

Before constructing an intersection product on the arithmetic Chow
groups of a stack, we must first construct a product on the
ordinary Chow groups of stacks over an arithmetic ring. As in
\cite{SGA-6} and \cite{Gillet-Soule-IHES-arith-intn-thy}, we shall replace the Chow groups of $\X$ by the
graded group associated to the $\gamma$-filtration on $K_0(\X)$.  However
there are some technicalities involved in doing this.

Recall that if $\mathfrak{X}$ is a Deligne-Mumford stack, then the
$G$-theory, with rational coefficients of $\mathfrak{X}$ is defined to
be
$$\mathbf{G}(\X):=\mathbf{H}_{et}(\mathfrak{X},\mathbf{G})$$
where $\mathbf{G}$ is the presheaf of $G$-theory spaces in the \'etale topology of $\mathfrak{X}$
associated to the Quillen $K$-theory of coherent sheaves of $\mathcal{O}_\X$-modules, and similarly the $K$-theory of $\X$ is
defined to be
$$\mathbf{K}(\X):=\mathbf{H}_{et}(\mathfrak{X},\mathbf{K})$$
where $\mathbf{K}$ is the presheaf of $K$-theory spaces associated
to the $K$-theory of locally free sheaves. In order to define the
$\gamma$-filtration we need to have $\lambda$-operations on the
sheaf $\mathbf{K}$. While the method of
\cite{Gillet-Soule-Filtrations-on-alg-kthy}, applied to any \'etale
hypercover  $V.\to\X$ by a simplicial regular scheme, should lead to
a construction of $\lambda$-operations, it is not clear that
the simplicial scheme $V.$ is $K$-coherent, in the terminology of
\emph{op. cit.}  However if we take as a model for $K$-theory the
presheaf of simplicial sets associated to the $G$-construction of
\cite{Gillet-Grayson-G-construction}, and the variations on the
G-construction described in \cite{Grayson-lambda-operations} applied
to the sheaf of categories $\mathcal{P}_\X$ associated to the
category of locally free sheaves of $\mathcal{O}_\X$-modules on the
\'etale site of $\X$, then we have, following \emph{op. cit}, maps
of simplicial sheaves
$$\lambda^k:\mathrm{Sub}_k G(\mathcal{P}_\X)\to G^{(k)}(\mathcal{P}_\X) $$
with both the domain and the target canonically weakly equivalent to
$G(\mathcal{P}_\X)$. It is shown in \textit{op. cit.} that if $R$ is
a commutative ring, then these maps induce the same maps on
$K$-theory as those of  \cite{Hillet-lambda-operations},  \cite{Kratzer-lambda-structure} and \cite{Soule-operations-CJM}.
\begin{thm}
Let $\X$ be a regular Deligne-Mumford stack.  Then the
$\gamma$-filtration on $K_0(\X)$ has finite length and there are
isomorphisms:
$$\mathrm{Gr}^p_\gamma(K_0(\X)) \simeq \mathrm{Gr}^p(K_0(\X))\simeq \CH^p(\X)_\mathbb{Q}.$$
Here $\mathrm{Gr}^\ast$ is the filtration by codimension of supports.
Furthermore,
if $\mathfrak{Y}\subset \mathfrak{X}$ is a closed substack of pure codimension $e$, then:
$$ Gr_\gamma^*K^{\mathfrak{Y}}_0(\mathfrak{X})\simeq \CH^{*-e}({\mathfrak{Y}})_\mathbb{Q}$$
\end{thm}
\begin{proof}
The coniveau spectral sequence for $\mathbf{K}(\mathfrak{X})$ has
$$E^{p,q}_1=\bigoplus_{\xi\in\mathfrak{X}^{(p)}}K_{q-p}(\xi)$$
where the direct sum is over punctual substacks of $\mathfrak{X}$.  As in \cite{Soule-operations-CJM},
consider the action of the Adams operations on the coniveau spectral sequence
(the Adams operations are linear combinations of the lambda operations).
If $\xi\in\mathfrak{X}$ is a point,
let $i:\xi\to\mathfrak{X}$ be the inclusion, and let $i^!$ be the ``sections with support'' functor.
Quillen's localization theorem (applied locally in the \'etale topology) implies that
we have a natural weak equivalence $Ri^!\mathbf{K}_\mathfrak{X}\simeq \mathbf{K}_\xi$.
Following the argument in \cite{Soule-operations-CJM}, this isomorphism shifts the weights of the
Adams operations by the codimension of $\xi$. We have isomorphisms
 $K_*(\xi)\simeq K_*(\mathbf{k}(\xi))$ (remembering that we are using $\Q$-coefficients),
 and just as in \emph{op. cit} we get that the coniveau spectral sequence is
 partially degenerate. Using the fact that $E_2^{p,p}\simeq\CH^p(\X)_\Q$, we then have the isomorphism of the theorem. The assertion
 about $K$-theory with supports follows from the same argument.
\end{proof}

The product structure on $K$-theory is compatible with the Adams operations (see \emph{op. cit}
and \cite{Gillet-Soule-Filtrations-on-alg-kthy}), and so the groups $\mathrm{Gr}^p_\gamma(K_0(\X))$ form
a contravariant functor from stacks to commutative graded rings. Via the isomorphism of the theorem, this gives
the Chow groups of stacks the same structure. We know from the appendix to exp. 0 of SGA6, \cite{SGA-6},
that if $X$ is a regular noetherian scheme, then the isomorphism
$ Gr_\gamma^*K_0({X})\simeq \CH^*(X)_\mathbb{Q}$
is compatible with the product on the Chow groups whenever it is defined.  Furthermore by the argument of
\cite{Gillet-Soule-Intersection-thy-adams-inventiones}, applied on an \'etale presentation of
$\mathfrak{X}$,
we know that if $\mathfrak{Y}\subset \mathfrak{X}$ and
$\mathfrak{Z}\subset \mathfrak{X}$ are integral substacks of codimension $p$ and $q$ respectively, intersecting properly, the intersection product
$[\mathfrak{Y}].[\mathfrak{Z}]$ computed using the
isomorphism $ Gr_\gamma^pK^{\mathfrak{Y}}_0(\mathfrak{X})\simeq \CH^0({\mathfrak{Y}})_\mathbb{Q}$ and
$ Gr_\gamma^qK^{\mathfrak{Z}}_0(\mathfrak{X})\simeq \CH^0({\mathfrak{Z}})_\mathbb{Q}$,
agrees with the product computed using Serre's intersection multiplicities.

Thus we have given another construction of the product structure on the Chow groups in the case of stacks over a field.

\subsection{Motives of stacks}
Let $\mathfrak{X}$ be a proper Deligne-Mumford stack over $S$. Recall from \cite{Gillet-Soule-Wt-Cplxs-Arith-Vartys}
that there exist proper hypercovers $\pi:X.\to\mathfrak{X}$ with the $X_i$ regular.
(We refer to these as non-singular proper hypercovers). Following \emph{op. cit.}, there is a covariant functor
$h$ from the category of regular projective schemes over $S$ to the category of $K_0$-motives
with rational coefficients.  Applying this functor to a proper hypercover $\pi:X.\to\mathfrak{X}$ we
get a chain complex of motives $h(X.)$.  It is proved in \emph{op. cit.} that if we have two
non-singular proper hypercovers $\pi_i:(X.)_i\to\mathfrak{X}$ for $i=1,2$, and a map
$f:(X.)_1\to(X.)_2$  of hypercovers (\ie $f\circ\pi_2=\pi_1$)
then the induced map of chain complexes of motives $h((X.)_1)\to h((X.)_2)$ is a
homotopy equivalence, and then that given any two nonsingular hypercovers $\pi_i:(X.)_i\to\mathfrak{X}$, irrespective
of whether there is a map between them, there is, nonetheless, a canonical isomorphism in the homotopy category of
chain complexes quasi-isomorphic to bounded complexes of $K_0$-motives: $h((X.)_1)\to h((X.)_2)$.
It follows upon applying the functor $K_0(\;)_\Q$, that we get, for each nonsingular proper hypercover
$\pi:X.\to\mathfrak{X}$ a cochain complex $K_0(X.)_\Q$ which is independent, up to homotopy equivalence, of the choice
of hypercover.

\begin{thm}\label{weight-complex-is-exact}
Suppose that the Deligne-Mumford stack $\mathfrak{X}$, flat and proper over $S$, is regular. Then for any
non-singular proper hypercover $\pi:X.\to\mathfrak{X}$, any regular variety $Y$ and any $p\geq 0$,
the augmentation
$$\pi_*:(i\mapsto G_p(X_i\times_S Y))\to G_p(\mathfrak{X}\times_S {Y})$$
is a quasi-isomorphism.  It follows by the argument of \cite{Gillet-Soule-Wt-Cplxs-Arith-Vartys}
that the associated complex of motives $h(X.)$ is exact in positive
degrees.
\end{thm}
\begin{proof}
As in \cite{Gillet-Soule-descent-motives}, we want to construct by induction $K_0$-correspondences
$h_i:X_i\to X_{i+1}$ for $i\geq -1$, (where we set $X_{-1}=\mathfrak{X}$)
which provide a contracting homotopy of the complex.
The slight complication here is that we need to start the induction with a stack. We could
extend the theory of correspondences and motives to stacks,
as was done in \cite{Toen_motives_for_DM-Stacks} for stacks over fields.  However we shall do something
more ad-hoc.  By lemma \ref{surjective}, we know that there is a class $\eta\in G_0(X_0)$ such that
$\pi_*(\eta)=[\mathcal{O}_{\mathfrak{X}}]$. It follows from the projection formula that if
$\alpha\in K_p(\mathfrak{X})$, then $\pi_*(\pi^*(\alpha)\gamma)=\alpha$, and furthermore, since we assume that
$\mathfrak{X}$ is flat over $S$, for any scheme $Y$ over $S$,
if $\alpha \in K_0(\mathfrak{X}\times_S Y)$, and $\pi_Y:X_0\times_S Y \to \mathfrak{X}\times_S Y $ is the induced map,
and $p:X_0\times_S Y \to X_0$ the projection, we have:
$$(\pi_Y)_*((\pi_Y)^*(\alpha)(\pi_Y)^*(\eta))=\alpha .$$
Thus we can view the class $\eta$ as determining a correspondence $[\eta]$ from $\mathfrak{X}$ to $X_0$, which induces a splitting
of the maps $\pi_*:G_*(X_0)\to G_*(\mathfrak{X})$.

Consider the diagram:
$$
\begin{CD}
X\times_\mathfrak{X}X @>q>> X \\
@VpVV                     @V\pi VV \\
X @>\pi >>   \mathfrak{X}
\end{CD}
$$
If $\mathcal{F}$ is a coherent sheaf on $X$, then by the projection formula:
$$\pi^*\pi_*([\mathcal{F}])= p_*(q^*(\alpha))[\mathcal{O}_X{}^L\otimes_{\mathcal{O}_\mathfrak{X}}\mathcal{O}_X],$$
and so $[\eta]\cdot\pi_*$ is induced by the correspondence
$\phi_*(\theta)q'^*(\eta)\in \ G_0(X\times_S X)$  where
$\theta=[\mathcal{O}_X{}^L\otimes_{\mathcal{O}_\mathfrak{X}}\mathcal{O}_X]$
where $\phi:X\times_\mathfrak{X}X\to X\times_S X$ is the natural map, and $q':X\times_S X\to X$ is projection onto the
second factor.  It is straightforward to check that $\phi_*(\theta)q'^*(\eta)$ is a
projector in the group of $K$-correspondences from $X$ to itself, and that
$(1_X\times \pi)_*(\phi_*(\theta)q'^*(\eta))=[\mathcal{O}_{\Gamma{\pi}}]\in \ G_0(X\times_S \mathfrak{X})$.
The proof now follows the argument  in \cite{Gillet-Soule-descent-motives}.
\end{proof}
\begin{cor}
The weight complex of a regular Deligne-Mumford stack proper over $S$ is isomorphic, in the
homotopy category of (homological) motives to a direct summand of the motive of a regular projective variety over $S$.
\end{cor}

\begin{remark}
For stacks over perfect fields follows immediately from (indeed is essentially equivalent to)
the conclusion of Toen in \cite{Toen_motives_for_DM-Stacks}.
\end{remark}

\section{Arithmetic Chow groups of proper stacks}\label{ch-hat-proper-stacks}
Given a regular Deligne-Mumford stack proper over $S$, we define its arithmetic Chow groups in exactly the same fashion
as in \cite{Gillet-Soule-IHES-arith-intn-thy} and \cite{Burgos-JAG-Arith-chow-rings}.
\begin{defn}
We define $\widehat{Z}^p(\X)$ to be the group consisting of pairs
$(\zeta,g)$ where $\zeta$ is cycle on $\X$, and
$g\in D^{p-1,p-1}/(Im(\partial)+Im(\widehat{\partial}))(\X_\mathbb{R})$
is
a green current for $\zeta$ in the sense of section \ref{Dolbeault}.

Here $\X_\mathbb{R}$ indicates that we are taking forms, or currents,
on $\X({\mathbb{C}})$ on which the complex conjugation $F_\infty$ on the base arithmetic ring acts by the rule of
\cite{Gillet-Soule-IHES-arith-intn-thy} and \cite{Burgos-JAG-Arith-chow-rings}.

If $\eta\in\X$ is a point, the Zariski closure of which has
codimension $p$ in $\X$, and $f\in\mathbf{k}(\eta)$, then $f$ pulls
back to a rational function  $\tilde{f}$ on the inverse image of the
Zariski closure of $\eta$ on any \'etale cover $U\to\X$, and hence
determines a $(p-1,p-1)$-current $-\log(|\tilde{f}|^2)$ on $U(\C)$.
Since this construction is local in the analytic, and hence in the
\'etale topologies, we get a current $-\log(|{f}|^2)$ on $\X(\C)$,
and we define $\widehat{\CH}^p(\X)$ to be the quotient of
$\widehat{Z}^p(\X)$ by the subgroup generated by cycles of the form
$(\mathrm{div(}f),-\log(|{f}|^2))$.
\end{defn}
The definition above makes sense for stacks which are not proper over $S$.  However for such a stack it will give ``naive'' $\chat$ groups as in
\cite{Gillet-Soule-IHES-arith-intn-thy} rather than the groups of \cite{Burgos-JAG-Arith-chow-rings}. As remarked in section 1,
we do not know a definition of a forms ``with logarithmic growth'' on a stack over $\mathbb{C}$ which is not proper,
and so we do not know how to modify the definition above to give the ``non-naive'' groups.

The proof of the following proposition is exactly the same as in the case of arithmetic schemes.
\begin{prop}\label{stack-exact-seq-prop} If $\X$ is proper over $S$,
There is an exact sequence
\begin{multline}\label{stack-exact-seq}
\mathrm{CH}^{p,p-1}(\mathfrak{X})\to H^{2p-1}_\mathcal{D}(\mathfrak{X}_\mathbb{R},\mathbb{R}(p))\\
\to\widehat{\mathrm{CH}}^p(\mathfrak{X})\to \mathrm{CH}^p(\mathfrak{X})\oplus{Z}^{p,p}(\mathfrak{X}_\mathbb{R})
\to\mathrm{H}^{2p}(\mathfrak{X}_\mathbb{R},\mathbb{R}(p))
\end{multline}
\end{prop}

On the other hand, if $\X$ is proper over $S$, it follows from theorem \ref{weight-complex-is-exact} and
 the discussion in section \ref{Dolbeault}, in particular
lemma \ref{inverse-limit-of-forms} that
 we have an exact sequence:
\begin{multline}\label{inv-limit-exact-seq}
\mathrm{CH}^{p,p-1}(\mathfrak{X})_\Q\to H^{2p-1}_\mathcal{D}(\mathfrak{X}_\mathbb{R},\mathbb{R}(p))\\
\to\lim_{\overset{\longleftarrow}{f:X\to\mathfrak{X}} }\widehat{\mathrm{CH}}^p({X})_\Q
\to \mathrm{CH}^p(\mathfrak{X})_\Q\oplus{Z}^{p,p}(\mathfrak{X}_\mathbb{R})
\to\mathrm{H}^{2p}(\mathfrak{X}_\mathbb{R},\mathbb{R}(p))
\end{multline}
where the inverse limit is over all proper surjective maps to $f:X\to\mathfrak{X}$ with nonsingular domain.

\begin{thm}
There is a canonical isomorphism
$$
\lim_{\overset{\longleftarrow}{f:X\to\mathfrak{X}} }\widehat{\mathrm{CH}}^p({X})_\Q
\to \widehat{\mathrm{CH}}^p(\mathfrak{X})_\Q  .
$$
\end{thm}
\begin{proof}
Since all but one term of the exact sequences \ref{stack-exact-seq}
and \ref{inv-limit-exact-seq} coincide, we have simply
to produce a map between the sequences.

In general if we are given a generically finite map ${f:X\to\mathfrak{X}}$, then we have push forward maps on cycles and currents which give
a push forward from $f_*:\chat^p(X)\to \widetilde{\CH^p(\X)}$, where $\widetilde{\CH^p(\X)}$ is defined in exactly the same way as ${\chat^p(\X)}$
\emph{except} that we do not require the form $dd^c(g)+\delta_\zeta=\omega$ be $C^\infty$. It is straight forward to check that this push
forward map gives a map  from the first exact sequence of proposition \ref{stack-exact-seq-prop} for $X$ to the corresponding exact
sequence for $\X$ in which $\chat^p(\X)$ is replaced
by $\widetilde{\CH^p(\X)}$  and $Z^{p,p}(\X)$ is replaced by the space of closed $(p,p)$-currents.

Suppose now that
$\alpha\in {\underset{f:X\to\mathfrak{X}}{\mathrm{lim}}}{\widehat{\mathrm{CH}}^q({X})}_\Q$.
Then for each
${f:X\to\mathfrak{X}}$, we have a class $\alpha_f$ which, by lemma \ref{inverse-limit-of-forms}, lies in
the subgroup of $\widehat{\mathrm{CH}}^p({X})_\Q$ mapping to
$f^*Z^{(p,p)}(\X(\C))$. Therefore if $f:X\to\X$ is generically
finite, since any $(p-1,p-1)$-current on $X$ pushes forward to a
$(p-1,p-1)$-current on $\X$, we can push forward
$\frac{1}{\deg(f)}\alpha_p$ to a class
$\frac{1}{\deg(f)}f_*(\alpha_f)$  in
$\widehat{\mathrm{CH}}^p({\X})_\Q$.

Given generically finite maps ${f:X\to\mathfrak{X}}$ and
${f':X'\to\mathfrak{X}}$, we can dominate both maps by a map
${f'':X''\to\mathfrak{X}}$ with $X''$ again  regular and $f''$
generically finite. Now applying the projection formula, we get
that $\frac{1}{\deg(f)}f_*(\alpha_f)$ is independent of $f$. Furthermore we now have a map  from the first exact
sequence of proposition \ref{stack-exact-seq-prop} for $X$ to the corresponding exact
sequence for $\X$.
\end{proof}

Since any map between stacks can be dominated by a map between non-singular hypercovers,
$\widehat{\mathrm{CH}}^*({\X})_\Q$ becomes a contravariant functor from stacks over $S$ to graded $\Q$-algebras.

If $\overline{E}=(E,\|\;\|)$ is a Hermitian vector bundle on $\X$, then $\overline{E}$  pulls back by any proper
surjective map $p:X\to\X$ to a Hermitian vector bundle on $X$, and the Chern classes $\widehat{C}_p(\overline{E})$
of \cite{Gillet-Soule-Annals}
define a class in
${\underset{p:X\to\mathfrak{X}}{\mathrm{lim}} }{\widehat{\mathrm{CH}}^p({X})}$,
and therefore in $\widehat{\mathrm{CH}}^p({\X})_\Q$, satisfying the properties of \emph{op. cit.}

\section{Arithmetic Chow groups of non-proper stacks}
The construction of the previous section assumed that the stack $\mathfrak{X}$ was proper over $S$.
Nonetheless we can make the following
\begin{defn}
If $\X$ is a regular, separated, Deligne-Mumford stack over $S$, we define the groups $\chat^*(\X)_\Q$ to be the inverse limit
of $\chat^*(X)_\Q$ over all proper surjective maps $\pi:X\to\mathfrak{X}$ with $X$ nonsingular.
\end{defn}
Note that this is equivalent to the inverse limit
over the category of non-singular proper hypercovers $\pi:V.\to \X$:
$$\chat^p(\X)_\Q:=\underset{{\pi:V.\to\X}}{\lim_{\longleftarrow}} \mathrm{H}^0(i\mapsto \chat^p(V_i)_\Q) .$$

To approach the basic properties of these groups, rather than working with the category of motives over $S$,
we shall introduce the category of (homological) motives
over $\mathfrak{X}$.

\subsection{Motives over a stack}
Let $\mathfrak X$ be a separated, regular Deligne-Mumford stack over
Spec$(\mathbb{Z})$. The category $\mathcal{V}ar_\mathfrak{X}$ of
varieties over $\mathfrak X$ has finite products and fibred
products. If $\alpha:X\to \X$, $\beta:Y\to\X$ and $\gamma:Z\to\X$
are objects, and $f:(X,\alpha)\to (Z,\gamma)$ and $g:(Y,\beta)\to
(Z,\gamma)$ are morphisms, in $\mathcal{V}ar_\mathfrak{X}$ we have
that the fibred product $X\times_ZY$ is the same whether we view
$f:X\to Z$ and $g:Y\to Z$ as morphisms, in
$\mathcal{V}ar_\mathfrak{X}$, or as morphisms in $\mathcal{V}ar_S$.
In particular if $f:X\to \mathfrak{X} $, $g:Y\to \mathfrak{X}$,
$h:Z\to \mathfrak{X}$ are objects in $\mathcal{V}ar_\mathfrak{X}$,
we have that
$$X\times_\mathfrak{X} Y\times_\mathfrak{X} Z\simeq (X\times_\mathfrak{X} Y)\times_Y(Y\times_\mathfrak{X}) Z.$$
Following the method of \cite{Gillet-Soule-Wt-Cplxs-Arith-Vartys},
we define the category of (homological) $K_0$-motives over
$\mathfrak{X}$ to be the idempotent completion of the category with
objects the  regular varieties which are projective over $\mathfrak{X}$,
and
$$\mathrm{KC}_\mathfrak{X}(X,Y):= G_0(X\times_\mathfrak{X}Y),$$
where $G_0$
is the Grothendieck group of coherent sheaves of
$\mathcal{O}_{X\times_\mathfrak{X}Y}$-modules (which as always we take with $\Q$-coefficients).
Note that $X \times_\X Y$ is not in general regular; however since $Y$ is regular,
${\mathcal O}_Y$ is of finite global tor-dimension, and hence
there is a bilinear product, given  $X$, $Y$ and $Z$ regular, projective, $\X$-varieties :
$$
* : G_0 (X \times_\X Y) \times G_0 (Y \times_\X Z) \to G_0 (X \times_\X Y \times_\X Z)
$$
$$
([{\mathfrak F}], [{\mathfrak G}]) \mapsto \sum_{i \geq 0} (-1)^i [ {\mathcal T}\!\!or_i^{{\mathcal O}_Y} ({\mathfrak F} , {\mathfrak G})] \, .
$$
Composing with the direct image map ($p_{XZ} : X \times_\X Y \times_Z Z \to X \times_\X Z$ being the natural projection):
$$
(p_{XZ})_* : G_0 (X \times_\X Y \times_\X Z) \to G_0 (X \times_Y Z)
$$
we get a bilinear pairing:
$$
G_0 (X \times_\X Y) \times G_0 (Y \times_\X Z) \to G_0 (X \times_\X Z)
$$
and hence:
$$
\mathrm{KC}_\X (X,Y) \times \mathrm{KC}_\X (Y,Z) \to \mathrm{KC}_\X (X,Z) \, .
$$
The proofs of the following theorems are straightforward, following the pattern in
\cite{Gillet-Soule-Wt-Cplxs-Arith-Vartys} and of the previous sections and
so we omit them.

\begin{lemma}
Given regular varieties $X,Y,Z$ and $W$ projective over $\X$, and elements
$\alpha \in \mathrm{KC}_\X (X,Y)$, $\beta \in \mathrm{KC}_\X (Y,Z)$, $\gamma \in \mathrm{KC}_\X (Z,W)$ we have
$$
\gamma \circ (\beta \circ \alpha) = (\gamma \circ \beta) \circ \alpha \, .
$$
\end{lemma}
Note that just as in the case of varieties over a scheme, if $\alpha:X\to \X$, $\beta:Y\to\X$ and $f:X\to Y$ is a morphism in
$\mathcal{V}ar_\mathfrak{X}$, then the graph of $f$ is a closed subscheme of $X\times_\X Y$ which is isomorphic to $X$,
and its structure sheaf defines a class $\Gamma(f)$ in $\mathrm{KC}_\X(X,Y)$.
and that it is straightforward to check:
\begin{lemma}
$$f\mapsto \Gamma(f)$$
defines a covariant functor $\Gamma:\mathcal{V}ar_\mathfrak{X}\to \KC_\X$.
\end{lemma}

\begin{prop}
The functors from regular varieties projective over $\mathfrak X$ to rational vector spaces, $G_i$ (Quillen $K$-theory of coherent sheaves,
with rational coefficients),
the associated graded for the $\gamma$-filtration on $G_0$, \emph{i.e} $\CH^*_\Q$, and real Deligne cohomology $X\mapsto \mathrm{H}^*_\mathcal{D}(X(\mathbb{C}),\mathbb{R}(*))$ all factor through $\KC_\X$.
\end{prop}

\begin{thm}
Suppose that $V.\to \X$ is a non-singular proper hypercover of the stack $\X$, then the complex $\Gamma(V.)$ is exact, except in degree zero, and
$\mathrm{H}_0(\Gamma(V.))$ is a direct summand of $\Gamma(V_0)$.  Furthermore, if $\pi:V.\to\X$ and $\pi':W.\to\X$ are two non-singular proper hypercovers of $\X$, and $f.:V.\to W.$
is a morphism over $\X$, then the induced map $\mathrm{H}_0(\Gamma(V.))\to \mathrm{H}_0(\Gamma(W.))$ is an isomorphism.
\end{thm}

\begin{cor} Suppose that $\X$ is a regular separated Deligne-Mumford stack over $S$, and $\pi:V.\to\X$ is non-singular proper hypercover.
There are natural isomorphisms, for all $p$ and $q$:
\begin{eqnarray*}
\mathrm{H}^p_\mathcal{D}(\X_\mathbb{R},\mathbb{R}(q))&\overset{\sim}{\longrightarrow}&
          \mathrm{H}^0(i\mapsto \mathrm{H}^p_\mathcal{D}((V_i)_\mathbb{R},\mathbb{R}(q)))\\
\pi^*:\CH^p(\X)_\Q&\overset{\sim}{\longrightarrow}&\mathrm{H}^0(i\mapsto \CH^p(V_i)_\Q)\\
\pi^*:\CH^{p,1}(\X)_\Q&\overset{\sim}{\longrightarrow}&\mathrm{H}^0(i\mapsto \CH^{p,1}(V_i)_\Q).
\end{eqnarray*}
Furthermore, these isomorphisms are compatible with both the cycle class maps
$$\gamma:\CH^p(-)_\Q\to \mathrm{H}^{2p}_\mathcal{D}(-,\mathbb{R}(p)))$$
and the regulator:
$$\rho:\CH^{p,1}(-)_\Q\to \mathrm{H}^{2p-1}_\mathcal{D}(-,\mathbb{R}(p))).$$
$$ $$

\end{cor}

\begin{thm}
The groups $\chat^*(\X)_\Q$ are contravariant with respect to $\X$ and have a natural product structure.
There are exact sequences:
\begin{multline*}
\mathrm{CH}^{p,p-1}(\mathfrak{X})_\Q\to H^{2p-1}_\mathcal{D}(\mathfrak{X}_\mathbb{R},\mathbb{R}(p))\\
\to\widehat{\mathrm{CH}}^p(\mathfrak{X})_\Q\to
\mathrm{CH}^p(\mathfrak{X})_\Q\oplus\widetilde{Z}_{\log}^{p,p}(\mathfrak{X}_\mathbb{R})
\to\mathrm{H}^{2p}_\mathcal{D}(\mathfrak{X}(\mathbb{C}),\mathbb{R}(p))
\end{multline*}
Here $\widetilde{Z}_{\log}^{p,p}(\mathfrak{X}_\mathbb{R})$ is the direct limit, over all proper surjective maps
$X\to\X$ of the groups ${Z}_{\log}^{p,p}({X}_\mathbb{R})$ of \cite{Burgos-Duke} and \cite{Burgos-JAG-Arith-chow-rings}.
\end{thm}

Unfortunately, I not know whether the complex $\widetilde{A}_{\log}^{*,*}(\mathfrak{X}(\mathbb{C}))$ which is the
direct limit over all proper surjective maps
$X\to\X$ of the logarithmic Dolbeault complexes ${A}_{\log}^{*,*}({X}(\mathbb{C}))$ of
\cite{Burgos-JAG-Arith-chow-rings} will compute the Dolbeault cohomology of $\X(\mathbb{C})$,
and thus the question of constructing a nice complex computing the real Deligne cohomology of
$\X(\mathbb{C})$ is open.

\end{document}